\documentclass[11pt]{article}
\usepackage{hyperref}
\usepackage{graphicx}
\usepackage{amssymb, amsthm}
\usepackage{epstopdf}
\usepackage{amsmath}
\usepackage{multicol}

\newcommand{\G}{Gr\"obner basis }
\newcommand{\Gs}{Gr\"obner bases }

\newcommand{\Gns}{Gr\"obner basis}
\newcommand{\lt}{{\rm lt}}
\newtheorem{thm}{Theorem}[section]

\begin{document}
\title{Fast Gr\"obner Basis Computation for Boolean Polynomials}
\author{ Franziska Hinkelmann$^{a,b}$ \and Elizabeth Arnold$^{c*}$ }

\maketitle
{\footnotesize
     \centerline{$^a$Department of Mathematics,
      Virginia Polytechnic Institute and State University,}
  \centerline{Blacksburg, VA 24061-0123, USA}
}
{\footnotesize
    \centerline{$^b$Virginia Bioinformatics Institute,
      Virginia Polytechnic Institute and State University,}
  \centerline{Blacksburg, VA 24061-0477, USA}
}
{\footnotesize
  \centerline{$^c$Department of Mathematics and Statistics, 
  James Madison University,}
  \centerline{Harrisonburg, VA 22807, USA}
}
{\footnotesize
  \centerline{$^*$Corresponding author: arnoldea@math.jmu.edu}
}

\begin{abstract}
We introduce the Macaulay2 package {\it BooleanGB}, which computes a Gr\"obner basis for
Boolean polynomials using a binary representation rather than symbolic.  We compare the runtime of
several Boolean models from systems in biology and give an application to Sudoku. 

\end{abstract}
\section{Introduction} \label{intro}

Buchberger's algorithm will theoretically compute a Gr\"obner basis for any
ideal of a multivariate polynomial ring with rational coefficients.  However,
due to memory constraints, many examples still cannot be computed in real
time.  Since all computations are done symbolically over the rational numbers,
two difficulties occur.  The coefficients of the polynomials during the
computation can grow very large, and also the degrees of the polynomials can
grow very large allowing a large number of monomials in each polynomial.  If
an ideal consists of Boolean polynomials in the quotient ring,  $QR = \mathbb{F}_2[x_1,\dots, x_n]/
\langle x_1^2 + x_1, x_2^2 + x_2, \dots x_n^2 + x_n \rangle$, we can modify
Buchberger's algorithm so that  these two issues are avoided entirely.  These
ideas have been implemented and discussed before in \cite{polybori,  Sato,
Bernasconi}. 
In this paper, we describe a fast \G implementation for Boolean polynomials in
{\it C++} which is included in the Macaulay2 release version 1.4 \cite{M2}. 

\section{Algorithm}
\subsection{Binary Representation} \label{binrep}

Consider a polynomial in the ring $QR = \mathbb{F}_2[x_1,\dots, x_n]/ \langle x_1^2 + x_1, x_2^2 + x_2, \dots x_n^2 + x_n \rangle$.  Since the Boolean polynomials are multilinear, we can represent a monomials in $n$ variables as a binary string of length $n$.   We store these strings as integers.  Polynomials, then, are just lists of these integers.  Arithmetic on the binary representation of these polynomials is very fast.  To add two monomials, we can just concatenate the binary strings.  If we add two identical monomials, they sum to zero.   To multiply two monomials, we use the binary operation OR.  Since $x_i^2 = x_i$ for every $i$, if $x_i$ appears in either monomial, then $x_i$ will appear in the product.   For the Boolean \G algorithm, we only need to multiply a monomial by a polynomial.  To do this, we just multiply the monomial by each monomial in the polynomial using bitwise OR.  Then we add the monomials by concatenation, removing duplicate pairs.

The other operations that we need for Buchberger's algorithm are least common multiple and division of monomials.  The least common multiple of two monomials is again just bitwise OR.  In fact, the operation of least common multiple in this Boolean ring is just multiplication.  To divide one monomial by another, we must first check divisibility.  If a monomial $m_1$ is divisible by a monomial $m_2$, then any variable appearing in $m_2$ must also appear in $m_1$.  So $m_1$ is divisible by  $m_2$ if and only if bitwise $m_1 - m_2 > 0$.  Once divisibility has been established, then division is binary exclusive OR.   Note that this is nothing more than subtraction.   In order to implement Buchberger's first criterion, we also need an operation to determine if two monomials are relatively prime.  This can be determined by comparing the minimum value for each variable in the monomial.  If the minimum for each variable is 0, then the monomials are relatively prime.  Now we have enough operations to execute Buchberger's algorithm.
\subsection{Buchberger's Algorithm} \label{BBA}
Given a set of Boolean polynomials, $F$ in $QR=\mathbb{F}_2[x_1,\dots, x_n]/
\langle x_1^2 + x_1, x_2^2 + x_2, \dots x_n^2 + x_n \rangle$, we want to
compute a \G  for the ideal generated by $F$  using Buchberger's Algorithm.
The basic Buchberger Algorithm computes $S-$polynomials for every pair of
polynomials in $F$.    It is important to note that when working in a quotient
ring, we need to perform Buchberger's algorithm on the ideal generated by $F$ together with the field polynomials.  For example, consider the ideal generated by $xy + z$ in $\mathbb{F}_2[x,y,z]/ \langle x^2 + x, y^2 + y, z^2 + z \rangle$.  Since there is just one polynomial in the ideal, Buchberger's algorithm produces the the \G for this ideal as  $xy + z$.  However, $xz + z$ is also in this ideal, since $xz + z = x(xy + z) + (xy + z)$.  In fact, the \G for the ideal generated by $xy+z$  in $\mathbb{F}_2[x,y,z]/ \langle x^2 + x, y^2 + y, z^2 + z \rangle$ is $\{xy + z, yz + z, xz+z\}$.  $xz + z$ is the reduction of the S-polynomial of $xy + z$ and $x^2 + x$.

So the first difficulty in implementing Buchberger's algorithm using the binary representation of Boolean polynomials is how to encode these quadratic field polynomials.  This difficulty is bypassed using the following result.\\

\begin{thm} The  S-polynomial of any multilinear Boolean polynomial and a quadratic field polynomial is always multilinear. \end{thm}
\begin{proof}  Let $f$ be a mutilinear Boolean polynomial with leading term $\lt(f)$.  If $\lt(f)$ and $x_i$ are relatively prime, then the S-polynomial of $f$ and $x_i^2 + x_i$ will reduce to 0 by Buchberger's first criterion \cite{buchberger}.  So we only consider quadratic field polynomials whose variable in the leading term is contained in $\lt(f)$.  Suppose $x_i$ is in $\lt(f)$.  Then
$S(f,x_i^2 +x_i) = f +(\frac{\lt(f)}{x_i})(x_i^2 +x_i) = f -\lt(f)+x_i$ which is multilinear.
\end{proof}

Therefore, no special encoding for quadratic field polynomials is necessary.

\subsection{Implementation}
The package {\it BooleanGB} contains the algorithm {\it gbBoolean}, which is an implementation of Buchberger's algorithm  for binary representations of Boolean polynomials over the quotient ring, 
$QR = \mathbb{F}_2[x1, x2, \ldots, x_n]/ \langle x_1^2 + x_1, x_2^2 + x_2, x_n^2 +
x_n\rangle$.  The algorithm is implemented in {\it C++} and is part of the Macaulay2
engine.  {\it gbBoolean}  computes a reduced Gr\"obner basis in lexicographic order for an
ideal of Boolean polynomials in $QR$.  All computations are done bitwise instead of symbolically, by
representing a monomial as bits in an integer and a polynomials as a list of
monomials.  On a 64 bit machine, the algorithm works in the ring with up to 64
indeterminates.  The algorithm uses Buchberger's first and second criteria \cite{buchberger} as described in
\cite[p. 109]{IVA}.  Here is algorithm computing the example from Section \ref{BBA}. 
\begin{verbatim}
i1 : loadPackage "BooleanGB";

i2 : R = ZZ/2[x,y,z];

i3 : QR = R/ideal(x^2-x, y^2-y, z^2-z);

i4 : I = ideal(x*y+z);

o4 : Ideal of QR

i5 : gbBoolean I

o5 = ideal (x*y + z, y*z + z, x*z + z)

o5 : Ideal of QR

\end{verbatim}

The code for {\it gbBoolean} is freely available with the source code distribution of Macaulay2 \cite{M2}.

\section{Applications}
Boolean \G algorithms such as {\it gbBoolean} can be used for any system of
Boolean polynomials.  In particular, exact cover problems, satisfiability
problems and problems in systems biology described below are examples.  In
some cases, it may be possible to turn a system of non-Boolean polynomials
into a Boolean system.  This involves dramatically increasing the number of
variables.  But in the example of Sudoku, described below, the time saved by
the bitwise computations outweighs the increase in variables.

\subsection{Boolean Models in Systems Biology}
Logical models \cite{GINsim} are widely used in systems biology.  They can be translated to
polynomial dynamical systems \cite{Alan}, in particular, logical models with binary
variables result in Boolean polynomials.  Key dynamic features of logical
models such as fixed points correspond to the points in algebraic varieties generated
by the polynomials describing the model.  To assure that {\it gbBoolean} is efficient
on ``practical'' ideals, we translate all binary logical models in the GINsim
repository \cite{GINsimRepo} to Boolean polynomial systems using ADAM
\cite{ADAM}, and compute the \G of the ideal describing the fixed
points.  For all logical models in the repository, {\it gbBoolean} is faster than current Macaulay2
implementations, run-times are are depicted in Fig. \ref{fig:benchmarks} and
table \ref{table:benchmarks}.

\subsection{Sudoku}
Sudoku is a popular game played on a $9 \times 9$ grid where the numbers 1-9
are filled in so that each number appears exactly once in each row column and
$3 \times 3$ block.  These constraints can be represented by polynomials, with
one variable for each cell in the grid.  (See \cite{ALT} and \cite{Gago} for
more details.)  For ease of demonstration, we use Shidoku as an example.  
Shidoku is played with the same rules as Sudoku, but on a $4 \times 4$ grid
with the numbers 1-4.

For Shidoku, we use 16 variables that can each take on only the values 1-4.
We can represent this fact, together with the constraints of the game, in a
total of 40 polynomials.    If we consider the ideal generated by these
polynomials in  $\mathbb{Q}[x_1, \dots, x_{16}]$ with the lexicographical
ordering, we cannot compute the \G for the ideal in Macaulay2 in a reasonable amount of
time.  We can, however, convert the problem to a system of Boolean polynomials
and use {\it gbBoolean} to compute the \Gns.

The Boolean system  has 64 variables,  representing each of the possible
values 1-4 for each of the 16 cells.  We consider the ideal of the polynomials
in $QR = \mathbb{F}_2[x1, x2, \ldots, x_n]/ \langle x_1^2 + x_1, x_2^2 + x_2, x_n^2 +
x_n\rangle$.  The Boolean system
gives us a total of 256 polynomials to represent the Shidoku constraints.  As
seen in Section \ref{results} below, {\it gbBoolean} computes the \G for the Boolean ideal in just 4.1 seconds.

\section{Results} \label{results}
We test {\it gbBoolean} on several random ideals and compare the run times with the standard \G implementation in Macaulay2.  In addition, we also test published Boolean logical models stemming from biological systems and compute the basis of the ideals corresponding to key dynamic features.  Finally, we test the algorithm on Sudoku, a problem that can be described by a system of Boolean polynomials.

We
compare {\it gbBoolean} to the symbolic computation in lexicographic order.
Since the graded reverse lexicographic (gRL) monomial order usually yields the 
fastest calculation,  we also compare {\it gbBoolean}  to
the combined times of computing a basis in gRL and lifting it to the quotient
ring with lexicographic order. 

Fig. \ref{fig:benchmarks} shows
the run-times on a 3.06 GHz Intel Core 2 Duo MacBook Pro.  In
almost all examples, {\it gbBoolean} is faster than current implementations in
Macaulay2. Cumulatively, {\it gbBoolean} is about four times faster
( $0.27614741\%$ ) than the \G
calculation in the ring with gRL order, and 50 times faster ( $0.017562079\%$ ) than the calculation in the
ring with lexicographic order.  
We do not list run-times for symbolic computations using the Sugar strategy
\cite{sugar} because it is much slower than the ``sugarless'' strategy for all
examples.

\begin{table}[ht]
\begin{center}
  \begin{tabular}{|c|c|c|c|}
  \hline & QR Lex & QR Lift & gbBoolean \\
  \hline
  II0 & 220.681 & 19.0562 & 4.4871\\
  \hline
  S2 & 0.002174 & 0.003237 & 0.00001\\
  \hline
  S3 & 0.0032 & 0.004926 & 0.000009\\
  \hline
  SS2 & 0.003438 & 0.005541 & 0.00001\\
  \hline
  Shidoku & 470.703 & 19.6797 & 4.11442 \\
  \hline
  TCR & 0.007104 & 0.012106 & 0.000208\\
  \hline
  THBool & 0.001946 & 0.003759 & 0.000009\\
  \hline
  BoolCC & 0.002007 & 0.003859 & 0.000012\\
  \hline
  erbb2 & 0.002609 & 0.003017 & 0.00001\\
  \hline
  yeastLi &0.002593 & 0.002013 & 0.00009 \\
  \hline
  \end{tabular}
  \caption{Run Times for Random, Biological, and Shidoku Examples}
  \label{table:benchmarks}
\end{center}
\end{table}

\begin{figure}[htb]
\centerline{ 
  \raise10pt
    \hbox { 
      \framebox{ \includegraphics[width=0.45\textwidth]{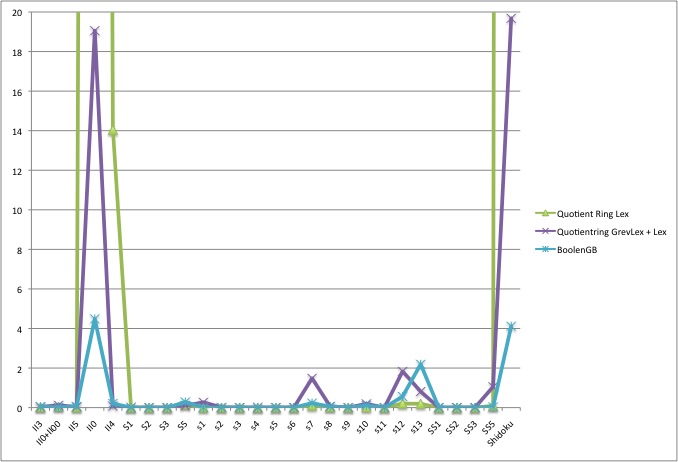} } 
      \qquad 
      \framebox{ \includegraphics[width=.45\textwidth]{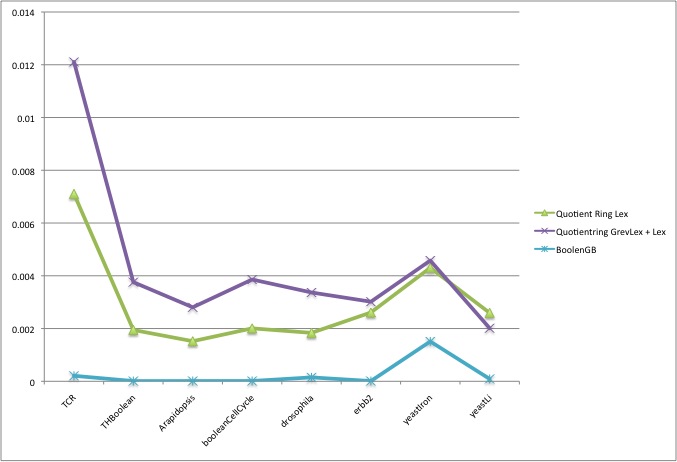}}
    } 
  } 
  \caption{Benchmarks:
  The left graph depicts run-times for random ideals and an
  application to Sudoku, the right graph for Boolean ideals
  originating from biological systems.}
\label{fig:benchmarks}
\end{figure}

\section{Conclusion}
Computing \Gs for Boolean polynomials bitwise rather than symbolically can
speed up computation time considerably.  In the example of Shidoku, even
though converting the problem from $\mathbb{Q}$ to $\mathbb{F}_2$ increases
the number of variables from 16 to 64, the time to compute the \G decreases.
The package {\it BooleanGB} in Macaulay2 has proven to be generally faster than the standard
algorithm in Macaulay2 over the quotient ring for Boolean polynomial ideals.  Further
improvements might be achieved by implementing the Sugar strategy \cite{sugar}
or other monomial orders. 

\small
\section*{Acknowledgments}
The authors thank Mike Stillman for adding the {\it C++}
implementation to the M2 engine and for his improvements of the
implementation.  Furthermore, the authors thank Samuel Lundqvist, Amelia Taylor,
Dan Grayson, and Reinhard Laubenbacher for stimulating discussions,
insightful comments, and organizing several Macaulay2 workshops.


\end{document}